\documentclass[12pt, reqno]{amsart}

\usepackage{
amsmath,
amssymb,
amsthm,
mathrsfs,
amsfonts,
ytableau,
enumitem,
upgreek,
soul,
yhmath,
mathtools,
shuffle,
kotex,
array,
caption,
}

\usepackage{latexsym}
\usepackage{stackengine}
\usepackage[colorlinks=true, pdfstartview=FitV, linkcolor=blue,citecolor=blue, urlcolor=blue]{hyperref}

\usepackage{tikz,tikz-cd}
\usepackage[colorinlistoftodos, textwidth = 2.3cm]{todonotes}

\ytableausetup{mathmode, boxsize=1.2em}
\title[Induction and restriction of two variable Hecke algebras]
{Induction and restriction of two variable Hecke algebras}

\author[W.-S. Jung]{Woo-Seok Jung}
\address{Department of Mathematics, Sogang University, Seoul 04107, Republic of Korea}
\email{jungws@sogang.ac.kr}

\author[Y.-T. Oh]{Young-Tak Oh}
\address{Department of Mathematics, Sogang University, Seoul 04107, Republic of Korea}
\email{ytoh@sogang.ac.kr}

\thanks{All authors were supported by the National Research Foundation of Korea (NRF) Grant funded by the Korean Government (NRF-2020R1F1A1A01071055).}


\newtheorem{theorem}{Theorem}[section]

\newtheorem{lemma}[theorem]{Lemma}
\newtheorem{corollary}[theorem]{Corollary}

\newtheorem*{claim*}{Claim}

\theoremstyle{definition}

\newtheorem{definition}[theorem]{Definition}
\newtheorem{remark}[theorem]{Remark}

\numberwithin{equation}{section} \numberwithin{figure}{section}
\numberwithin{table}{section}

\newcommand{\nc}{\newcommand}

\nc{\SG}{\mathfrak{S}}
\nc{\frakR}{\mathfrak{R}}
\nc{\frakL}{\mathfrak{L}}
\nc{\PCT}{\mathrm{PCT}}
\nc{\SPCT}{\mathrm{SPCT}}
\nc{\RT}{\mathrm{RT}}
\nc{\SRT}{\mathrm{SRT}}
\nc{\RCT}{\mathrm{RCT}}
\nc{\SRCT}{\mathrm{SRCT}}
\nc{\SYRT}{\mathrm{SYRT}}
\nc{\SYCT}{\mathrm{SYCT}}
\nc{\SPYCT}{\mathrm{SPYCT}}
\nc{\tst}{\mathtt{st}}
\nc{\Span}{\mathrm{span}}
\nc{\comp}{\mathrm{comp}}
\nc{\rmst}{\mathrm{st}}
\nc{\Des}{\mathrm{Des}}
\nc{\set}{\mathrm{set}}
\nc{\wt}{\mathrm{wt}}
\nc{\ch}{\mathrm{ch}}
\nc{\id}{\mathrm{id}}
\nc{\Sym}{\mathrm{Sym}}
\nc{\Qsym}{\mathrm{QSym}}
\nc{\Nsym}{\mathrm{NSym}}
\nc{\sh}{\mathrm{sh}}
\nc{\bfS}{\mathbf{S}}
\nc{\bfm}{\mathbf{m}}
\nc{\hbfS}{\widehat{\mathbf{S}}}
\nc{\bfF}{\mathbf{F}}
\nc{\calB}{\mathcal{B}}
\nc{\calS}{\mathcal{S}}
\nc{\hcalS}{\widehat{\mathcal{S}}}
\nc{\alphamax}{\alpha_{\rm max}}
\nc{\brho}{\overline{\rho}}
\nc{\bphi}{\overline{\phi}}
\nc{\calV}{\mathcal{V}}
\nc{\calR}{\mathcal{R}}
\nc{\sfR}{\mathsf{R}}
\nc{\calG}{\mathcal{G}}
\nc{\tal}{\lambda(\alpha)}
\nc{\tbe}{\widetilde{\beta}}
\nc{\opi}{\overline{\pi}}
\nc{\calP}{\mathcal{P}}
\nc{\rmtop}{\mathrm{top}}
\nc{\rad}{\mathrm{rad}}
\nc{\bfP}{\mathbf{P}}
\nc{\SET}{\mathrm{SET}}
\nc{\SIT}{\mathrm{SIT}}
\nc{\rev}{\mathrm{r}}
\nc{\Th}{\theta}
\nc{\mPhi}{\Phi}
\nc{\mphi}{\phi}
\nc{\mPsi}{\Psi}
\nc{\hmPsi}{\widehat{\Psi}}
\nc{\mpsi}{\psi}
\nc{\mGam}{\Gamma}
\nc{\tcd}{\mathtt{cd}}
\nc{\trd}{\mathtt{rd}}
\nc{\trcd}{\mathtt{rcd}}
\nc{\rmr}{\mathrm{r}}
\nc{\rmc}{\mathrm{c}}
\nc{\rmt}{\mathrm{t}}
\nc{\bubact}{\,\scalebox{0.6}{$\bullet$}\,}
\nc{\hbubact}{\,\scalebox{0.6}{$\widehat{\bullet}$}\,}
\nc{\col}{\rm col}
\nc{\row}{\rm row}
\nc{\calE}{\mathcal{E}}
\nc{\calT}{\mathscr{T}}
\nc{\sfT}{\mathsf{T}}
\nc{\calEsa}{\mathcal{E}^\sigma(\alpha)}
\nc{\tauC}{\tau_{\scalebox{0.5}{$C$}}}
\nc{\sytabC}{\sytab_{\scalebox{0.5}{$C$}}}
\nc{\bbfP}{\overline{\bfP}}
\nc{\pr}{\mathbf{pr}}
\nc{\Ups}{\Upsilon}
\nc{\pact}{\diamond}
\nc{\tauE}{\tau_{\scalebox{0.5}{$E$}}}
\nc{\tauF}{\tau_{\scalebox{0.5}{$F$}}}
\nc{\tauG}{\tau_{\scalebox{0.5}{$G$}}}
\nc{\rtE}{T_{\scalebox{0.5}{$E$}}}
\nc{\rtF}{T_{\scalebox{0.5}{$F$}}}
\nc{\rtG}{T_{\scalebox{0.5}{$G$}}}
\nc{\oPaE}{\overline{\Phi}_{\alpha_E}}
\nc{\oPaF}{\overline{\Phi}_{\alpha_F}}
\nc{\oPaG}{\overline{\Phi}_{\alpha_G}}
\nc{\tab}{\tau}
\nc{\sytab}{\widehat{\tau}}
\nc{\hatE}{\widehat{E}}
\nc{\hati}{\hat{i}}
\nc{\hcalE}{\widehat{\calE}}
\nc{\hatC}{\widehat{C}}
\nc{\bal}{{\boldsymbol{\upalpha}}}
\nc{\bbe}{{\boldsymbol{\upbeta}}}

\nc{\bgam}{{\boldsymbol{\upgamma}}}
\nc{\bdel}{{\boldsymbol{\updelta}}}
\nc{\weakcon}{\odot}
\nc{\basisI}{I}

\nc{\ldalpha}{\lambda(\alpha)}
\nc{\SRIT}{\mathrm{SRIT}}
\nc{\re}{\mathrm{rev}}
\nc{\otau}{\overline{\tau}}
\nc{\rtop}{{\rm top}}
\nc{\sfc}{\mathsf{c}}
\nc{\sfr}{\mathsf{r}}
\nc{\tH}{\mathtt{H}}
\nc{\tV}{\mathtt{V}}
\nc{\rpi}{\mathring{\pi}}
\nc{\cpi}{\check{\pi}}
\nc{\frakm}{\mathfrak{m}}
\nc{\fke}{\mathfrak{e}}
\nc{\Hom}{\mathrm{Hom}}
\nc{\module}{\mathrm{mod} \, }
\nc{\fmodule}{\mathrm{fmod} \, }

\nc{\SPCTsa}{\SPCT^\sigma(\alpha)}
\nc{\bfSsa}{\bfS_\alpha^\sigma}
\nc{\bfSsaC}{{\bfS}^\sigma_{\alpha,C}}
\nc{\hbfSsa}{\widehat{\bfS}_\alpha^\sigma}
\nc{\upineq}{\rotatebox{90}{$<$}}
\nc{\downineq}{\rotatebox{270}{$<$}}
\nc{\diagineq}{\rotatebox{135}{$<$}}
\nc{\frakB}{\mathfrak{B}}
\nc{\hxi}{\widehat{\xi}}
\nc{\hxidwJ}{\hxi_{\scalebox{0.55}{$J$}}}
\nc{\hxiupJ}{\hxi^{\scalebox{0.55}{$J$}}}
\nc{\scrS}{\mathscr{S}}
\nc{\bfT}{\mathbf{T}}

\nc{\ra}{\rightarrow}

\nc{\matr}[2]{\left( \hspace{-1ex} \begin{array}{c} #1 \\ #2 \end{array} \hspace{-1ex} \right)}

\definecolor{wsgreen}{rgb}{0,0.5,0}

\nc{\DIRT}{\mathrm{DIRT}}
\nc{\hpi}{\pi}
\nc{\frakI}{\mathfrak{I}}
\nc{\hfrakI}{\widehat{\mathfrak{I}}}
\nc{\orho}{\overline{\rho}}
\nc{\autotheta}{\uptheta}
\nc{\autophi}{\upphi}
\nc{\autochi}{\upchi}
\nc{\autoomega}{\upomega}
\nc{\hIM}{\widehat{\frakB}}
\nc{\bfpi}{\boldsymbol{\uppi}}
\nc{\bfopi}{\overline{\boldsymbol{\uppi}}}
\nc{\ofrakB}{\overline{\frakB}}
\nc{\rmw}{\mathrm{w}}
\nc{\ostar}{\;\overline{*}\;}
\nc{\rank}{\mathrm{rank}}
\nc{\fkp}{\mathfrak{p}}
\nc{\bfR}{\mathbf{R}}

\nc{\upsig}{{\boldsymbol{\upsigma}}}
\nc{\bfSsaE}{{\bfS}^\upsig_{\alpha,E}}

\nc{\hfkp}{\widehat{\mathfrak{p}}}

\nc{\hautophi}{{\widehat{\autophi}}}
\nc{\hautotheta}{{\widehat{\autotheta}}}
\nc{\hautoomega}{{\widehat{\autoomega}}}
\nc{\rmperm}{\mathrm{perm}}

\nc{\bfsigJ}{\boldsymbol{\sigma}_{\scalebox{0.55}{$J$}}}
\nc{\bfrhoJ}{\boldsymbol{\rho}^{\scalebox{0.55}{$J$}}}

\nc{\pistar}[1]{\pi_{#1}^*}
\nc{\wfkp}{\widetilde{\mathfrak{p}}}
\nc{\bfpsi}{\boldsymbol{\uppsi}}

\nc{\yt}[1]{\todo[size=\tiny,color=blue!10]{#1 \\ \hfill --- Young-Tak}}
\nc{\YT}[1]{\todo[size=\tiny,inline,color=blue!10]{#1
		\\ \hfill --- Young-Tak}}
		
\nc{\yh}[1]{\todo[size=\tiny,color=cyan!10]{#1 \\ \hfill --- Young-Hun}}
\nc{\YH}[1]{\todo[size=\tiny,inline,color=cyan!10]{#1
		\\ \hfill --- Young-Hun}}
		
\nc{\sy}[1]{\todo[size=\tiny,color=magenta!10]{#1 \\ \hfill --- So-Yeon}}
\nc{\SY}[1]{\todo[size=\tiny,inline,color=magenta!10]{#1
		\\ \hfill --- So-Yeon}}
		
\nc{\ws}[1]{\todo[size=\tiny,color=green!10]{#1 \\ \hfill ---  Woo-Seok}}
\nc{\WS}[1]{\todo[size=\tiny,inline,color=green!10]{#1
		\\ \hfill --- Woo-Seok}}

\definecolor{purple}{rgb}{0.44, 0.0, 1.0}

\newenvironment{red}{\relax\color{red}}{\hspace*{.5ex}\relax}
\newenvironment{blue}{\relax\color{blue}}{\hspace*{.5ex}\relax}
\newenvironment{green}{\relax\color{wsgreen}}{\hspace*{.5ex}\relax}
\newenvironment{magenta}{\relax\color{magenta}}{\hspace*{.5ex}\relax}
\newenvironment{purple}{\relax\color{purple}}{\hspace*{.5ex}\relax}

\nc{\ber}{\begin{red}}
\nc{\er}{\end{red}}
\nc{\beb}{\begin{blue}}
\nc{\eb}{\end{blue}}
\nc{\bema}{\begin{magenta}}
\nc{\ema}{\end{magenta}}
\nc{\begr}{\begin{green}}
\nc{\egr}{\end{green}}

\nc{\bepu}{\begin{purple}}
\nc{\epu}{\end{purple}}

\begin{document}

\maketitle

\begin{abstract}
The purpose of this paper is to study induction and restriction of two-variable Hecke algebras.
First, we provide the explicit form of the Mackey decomposition formula. 
And then, we elucidate how (anti-)involutions interact with induction product and restriction.
\end{abstract}

\section{Introduction} 

Two-variable Hecke algebras are introduced as specializations of 
Iwahori-Hecke algebras (see \cite[Section 4.4]{GP00} or \cite[Section 7]{90H}).
Let $R$ be an unital commutative ring and $(W, S)$ a finite Coxeter system. 
Suppose we are given parameters $a_s, b_s\in R \,\,(s\in S)$ subject only to the requirement that $a_s=a_t$ and $b_s=b_t$ whenever $s$ and $t$ are conjugate in $W$.
Following Geck and Pfeiffer~\cite[Definition 4.4.1]{GP00}, 
one can associate to $(W, S)$ 
the {\it Iwahori-Hecke algebra} $H_R(W,S,\{a_s, b_s: s\in S\})$ over $R$.
For $a,b\in R$, the {\it two variable Hecke algebra} $H_W(a, b)$ over $R$
is defined to be the generic algebra $H_R(W,S,\{a_s, b_s: s\in S\})$ with the parameters
$a_s=a, b_s=b$ for all $s\in S$.
When $R$ is the field of complex numbers, it was proven in~\cite[Propsition 6.1]{BLL12} that $H_W(a, b)$ is isomorphic to 
one of the following families of algebras:
\begin{itemize}
    \item Hecke algebras at a generic value,
    \item Hecke algebras at a root of unity,
    \item $0$-Hecke algebras,
    \item nil-Coxeter algebras. 
\end{itemize}
This result was originally stated only for finite Coxeter groups of type $A$, 
but one can extend it to arbitrary finite type in the exactly same way as in~\cite{BLL12}.

The first objective of this paper is to provide the Mackey decomposition formula of two variable Hecke algebras.
Let us explain the motivation of this problem.
In 2009, Bergeron and Li~\cite[Section 3.1]{09BL} presented an axiomatic definition of a {\it tower of algebras}, 
which is a graded $\mathbb C$-algebra $A = (\oplus_{n\ge 0}A_n, \rho)$
satisfying certain five axioms.
Then they proved that $G_0(A):=\oplus_{n\ge 0}G_0(A_n)$ and $K_0(A):=\oplus_{n\ge 0}K_0(A_n)$ 
have a bialgebra structure and are dual to each other as connected graded bialgebras,
where $G_0(A_n)$ (resp. $K_0(A_n)$) is the Grothendieck group of the category of all finitely generated $A_n$-modules (resp. projective $A_n$-modules).
In their proof, the fifth axiom, which is an analogue of Mackey’s formula for $G_0(A)$ and $K_0(A)$, plays an important role.
At this point, one might naturally ask if Bergeron and Li's analogue of Mackey’s formula can be lifted to the module level in the same format, 
which is nontrivial unless the category of finitely generated $A_n$-modules is not semisimple for all $n\ge 0$. 
For instance, the category of finitely generated $H_n(0)$-modules is known to be not semisimple for $n\ge 3$
(see~\cite[Section 4.3]{02DHT} or~\cite[Theorem 3.1]{11DY}).
In~\cite[Proposition 9.1.8]{GP00}, 
the Mackey decomposition formula was stated in the case where 
given parameters $a_s, b_s\in R$
satisfy that $b_s=1-a_s$ for all $s \in S$ 
and the $H_R(W,S,\{a_s, 1-a_s: s\in S\})$-module in consideration is 
free as an $R$-module.

In~\cite[Theorem 3]{21JKLO}, this formula was provided   
for the family of $0$-Hecke modules called {\it weak Bruhat interval modules}.
In either case, the question posed above turns out to be affirmative. 
The present paper deals with this question for 
arbitrary modules of arbitrary two-variable Hecke algebras.
As stated before, when $R=\mathbb C$, our algebras in consideration 
cover Hecke algebras at a generic value,
Hecke algebras at a root of unity,
$0$-Hecke algebras, and 
nil-Coxeter algebras. 
Our Mackey decomposition formula is stated for parabolic subalgebras of $H_W(a, b)$ 
and its explicit form is provided (Theorem~\ref{thm: Mackey} and Corollary~\ref{cor: Mackey for 0-Hecke}).

The second objective of this paper is to show how (anti-)involutions interact with induction product and restriction.
In~\cite[Proposition 3.2]{05Fayers}, Fayers introduced involutions $\autophi, \autotheta$ and an anti-involution $\autochi$
on the $0$-Hecke algebra $H_n(0)$ defined by \[\autophi(\pi_s)=\pi_{w_0 s w_0}, \quad \autotheta(\pi_s)=1-\pi_s, \quad  \autochi(\pi_s)=\pi_s  \quad (s \in S),\]
where $w_0$ is the longest element of the symmetric group $\mathfrak S_n$, and used them to derive a branching rule that describes the submodule lattice of $M \big\uparrow^{H_{n}(0)}_{H_{n-1}(0)}$ 
for any simple module $M$ of $H_{n-1}(0)$.
In~\cite[Section 3.4]{21JKLO}, the (anti-)involution-twists of weak and negative weak Bruhat interval modules are intensively investigated for the (anti-)involutions obtained by composing $\autophi,\autotheta$ and $\autochi$. 
It is quite interesting to note that 
they exhibit very strong symmetry (see~\cite[Table 1 and Table 2]{21JKLO}).
Here, mimicking~\cite{05Fayers}, we introduce two involutions  
and one anti-involution on $H_W(a,b)$
(Lemma~\ref{involutions and anti-involution}).
Then we show how $H_W(a,b)$-modules behave under the process of (anti-)involution-twist for all the (anti-)involutions obtained by composing these (anti-)involutions (Theorem~\ref{thm: automorhpism, induction, restriction}
and Theorem~\ref{thm: chi twist, induction}).
In particular, Theorem~\ref{thm: chi twist, induction} can be viewed as a generalization of Fayers' lemma~\cite[Lemma~6.4]{05Fayers}.

This paper is organized as follows.
In Section~\ref{Sec. Two variable Hecke algebras}, we introduce the prerequisites on the two variable Hecke algebras including 
induction, restriction, and (anti-)automorphism twists.
In Section~\ref{Sec.The Mackey formula for two variable Hecke algebras},
we provide the Mackey decomposition formula for two variable Hecke algebras.
Section~\ref{Sec. Interaction with induction product and restriction} is devoted to elucidating the comparability of (anti-)-automorphism twists
with induction product and restriction.

\section{Two variable Hecke algebras}\label{Sec. Two variable Hecke algebras}
Let $(W,S)$ be a finite Coxeter system with $S=\{s_1, s_2,\ldots,s_n\}$. That is, 
$$
W = <s_1, s_2,\ldots,s_n \mid (s_i s_j)^{m_{ij}}>
$$
for some symmetric $n$ by $n$ matrix $(m_{ij})$ with entries $\{2,3,4,6\}$ with $m_{ii}=1$ and $m_{ij}>1$ for $i \neq j$.
For $I \subseteq S$, we write $W_I$ for the subgroup of $W$ generated by $I$, which is also a Coxeter group and called a \emph{parabolic subgroup} of $W$.

Throughout this paper, $R$ denotes a commutative ring with $1$.

\begin{definition}
For $a,b\in R$, the \emph{two variable Hecke algebra} $H_W(a, b)$ of $W$ over $R$ is an $R$-algebra with $1$ generated by the elements $\pi_s$ ($s \in S$) subjects to the following defining relations:
\begin{align*}
    \pi_s \pi_s &= a \pi_s + b, \\
    (\pi_i \pi_j \pi_i \ldots)_{m_{ij}} &= (\pi_j \pi_i \pi_j \ldots)_{m_{ij}}\quad \text{for all $i,j\in S$ with $i\ne j$},
\end{align*}
where the notation $(aba\cdots)_{m}$ denotes the product of $m$ terms following the order in the parenthesis.
\end{definition}

When it is not necessary to specify $a$ and $b$, 
we frequently write $H_W$ for $H_W(a, b)$ for simplicity.
For any reduced expression $s_{i_1}s_{i_2} \ldots s_{i_l}$ of $w \in W$, let $\pi_w := \pi_{i_1}\pi_{i_2} \ldots \pi_{i_l}$. It is well known that $\pi_w$ is independent of the choices of reduced expressions (Matsumoto's theorem).
Given $I \subseteq S$, let $W_I$ be the parabolic subgroup of $W$ corresponding to $I$.
The two variable Hecke algebra $H_{W_I}(a, b)$ associated to the Coxeter system $(W_I,I)$ is
a subalgebra of $H_W(a, b)$, which is called the \emph{parabolic subalgebra} of $H_W(a, b)$ corresponding to $I$.

\begin{theorem}{\rm (\cite[Section 7.1]{90H})}\label{thm: left-right actions}
For $a,b\in R$, the two variable Hecke algebra $H_W(a, b)$ over $R$ is a free $R$-module with basis elements $\pi_w \, (w \in W)$,  and for all $s \in S, w \in W$, multiplication is 
given by
\[\pi_s \pi_w
=\begin{cases}
\pi_{sw}          &\text{ if } \ell(sw)>\ell(w), \\ 
a\pi_w+ b\pi_{sw} & \text{ if }\ell(sw)<\ell(w).
\end{cases}\]
\end{theorem}
For later use, we note that 
\[\pi_w \pi_s
=\begin{cases}
\pi_{ws}          &\text{ if } \ell(ws)>\ell(w), \\ 
a\pi_w+ b\pi_{ws} & \text{ if }\ell(ws)<\ell(w)
\end{cases}\]
and $\pi_v \pi_w = \pi_{vw} \text{ if }  \ell(v)+ \ell(w) = \ell(vw)$ for $v,w \in W$.

The two variable Hecke algebras were studied intensively in~\cite[Section 6]{BLL12}
in the case where $R = \mathbb{C}$ and $W = \SG_n$. 
Replacing $\SG_n$ with $W$ in the proof of~\cite[Proposition 6.1]{BLL12} gives the following  classification.

\begin{theorem}{\rm (cf.~\cite[Proposition 6.1]{BLL12})}\label{thm: hecke algebra specialization}
Let $R = \mathbb{C}$ and $a,b \in \mathbb C$.
Then $H_W(a,b)$ is isomorphic to one of the following four families of algebras:
\begin{itemize}
    \item a Hecke algebra $H_W(\nu)$ at a generic value of $\nu$, or
    \item a Hecke algebra $H_W(\xi)$ at a root $\xi$ of unity , or
    \item the $0$-Hecke algebra $H_W(0)$ (when $a \ne 0$ but $b=0$), or
    \item the nil-Coxeter algebra $N_W$ (when $a=b=0$).
\end{itemize}
\end{theorem}

For an $R$-algebra $A$, 
let $\module(A)$ be the category of left $A$-modules.
Let $M \in \module(H_{W_I})$ for some $I \subseteq S$. 
Then the \emph{induction} of $M$ to $H_W$ is defined by
$$
M \big\uparrow ^{H_W}_{H_{W_I}} := H_W \underset{H_I}{\otimes}M,
$$
where we are viewing $H_W$ as an $(H_W, H_{W_I})$-bimodule.
If $N \in \module(H_W)$, we write the corresponding $H_{W_I}$-module as 
$$
N \big\downarrow ^{H_W}_{H_{W_I}}
$$
and call it the \emph{restriction} of $N$ to $H_{W_I}$.

For an algebra homomorphism $f: A \rightarrow B$ and $M \in \module(B)$, the following defines a left $A$-module structure on $M$.
\begin{equation}\label{symbol for action}
a\cdot^{f} m := f(a) \cdot m
\end{equation}
Thus, $f$ induces a functor $F: \module(B) \rightarrow \module(A)$. For example, the functor induced by the inclusion $\iota: H_{W_I} \rightarrow H_W$ is the restriction functor.

We close this section by introducing the notion of (anti-)automorphism twists.
Let $\mu: B \to A$ be a morphism of associative algebras. Given a $A$-module $M$, we define $\mu[M]$ be a $B$-module with the same underlying vector space as $M$ and with the action $\cdot^{\mu}$ twisted by $\mu$ in such a way that 
\begin{align*}
b \cdot^{\mu} v := \mu(b) \cdot v \quad \text{for $a \in A$ and $v \in M$}.
\end{align*}
This induces a covariant functor 
\[
F_{\mu}: \module A \ra \module B, \quad M \mapsto \mu[M],
\]
where $F_{\mu}(h):\mu[M] \to \mu[N], m \mapsto h(m)$ for every $A$-module homomorphism $h: M\to N$. 
We call $F_{\mu}$ the \emph{$\mu$-twist}.
For example, the functor induced by the inclusion $\iota: H_{W_I} \rightarrow H_W$ is the restriction functor.

Similarly, given an anti-homomorphism $\nu: B \to A$ of associative algebras,
we define $\nu[M]$ to be the $B$-module with $M^*$, the dual space of $M$, as the underlying space and with the action $\cdot^{\beta}$ defined by
\begin{align}\label{eq: anti-automorphism twist}
(a \cdot^{\nu} \delta) (v) := \delta(\nu (a) \cdot v) \quad \text{for $a \in A$, $\delta \in M^*$, and $v \in M$}.
\end{align}
Any anti-automorphism $\nu$ of $A$ induces a contravariant functor
\[
G_{\nu}: \module A \ra \module B, \quad M \mapsto \nu[M], 
\]
where $G_{\nu}(h):\nu[N] \to \nu[M], \delta \mapsto \delta \circ h$
for every $A$-module homomorphism $h: M\to N$.
We call $G_{\nu}$ the \emph{$\nu$-twist}.

\section{The Mackey decomposition formula for two variable Hecke algebras}
\label{Sec.The Mackey formula for two variable Hecke algebras}
Let $(W, S)$ be a finite Coxeter system. We denote by $W^I$ (resp. ${}^{I}W$) 
the set of minimal length left (resp. right) coset representatives of $W_I$ in $W$.

For any $w \in W$, it is well known that $w$ can be uniquely factorized as $xy$ for $x \in W^I$ and $y \in W_I$ (for instance, see~\cite[Proposition 2.4.4]{05BB}). 
In this case, it holds that  $\ell(w) = \ell(x)+\ell(y)$. 
Let us call this factorization as the \emph{reduced factorization of $w$ with respect to $W/W_I$}.
The reduced factorization with respect to $W_I \backslash W$ can be defined in a similar manner.

For $I, J \subseteq S$, let $C$ be a $(W_J, W_I)$-double coset in $W_J \backslash W / W_I$. 
It is well known that there exists a unique minimal length element $\tau$ in $C$. 
Let ${}^{J}W^{I} $ be the set of minimal length representatives of all the double cosets in $W_J \backslash W / W_I$, which has the following nice characterization.

\begin{lemma}{\rm (\cite[Chapter 4]{B02ch4-6})}  \label{lem: minimal representatives of double coset} For $I, J \subseteq S$,
$$
{}^{J}W^{I} = {}^{J}W \cap W^{I} = \{w \in W \mid J \subseteq Asc_{L}(w) \text{ and } I \subseteq Asc_{R}(w) \}
$$
where $Asc_{L}(w):= \{s \in S \mid \ell(sw) = \ell(w) +1 \}$ and $Asc_{R}(w):= \{s \in S \mid \ell(ws) = \ell(w) +1 \}$.
\end{lemma}

For $I, J \subseteq S$ and $w \in W$, let
\begin{align*}
    K(w) &:= J \cap (w I w^{-1}), \\
    w^{-1} K(w) w &:= (w^{-1}Jw) \cap I.
\end{align*}

\begin{lemma}{\rm (\cite[Proposition 8.3]{GS84}, \cite[Theorem 1.2]{C85}, \cite[Corollary 3]{BKPST18}, \cite[Section 1.10]{90H})}
\label{lem: triple reduced expression}
Let $I, J \subseteq S$ and $\tau \in {}^{J}W^{I}$.
Then, for any element $u \in W_J$, $u\tau \in W^I$ if and only if $u \in W^{K(\tau)}_J$.
Consequently, every element of $W_J \tau W_I$ can be written uniquely as $u \tau v$, where $u \in W^{K(\tau)}_J, v \in W_I$, and $\ell(u \tau v)=\ell(u)+\ell(\tau)+\ell(v)$.
\end{lemma}

From now on, when we are dealing with two variable Hecke algebras associated with paraboic subalgebras, 
we simply write $H_I$ for $H_{W_I}(a, b)$ for each $I \subseteq S$.
The subsequent Lemma plays a key role in the proof of our main result.

\begin{lemma}\label{lem: conjugate homomorphism}
Let $I, J \subseteq S$ and $w \in W$.  
Then the map 
\[c_w:H_{K(w)} \rightarrow
H_{w^-1 K(w) w}, \quad \pi_{k} \mapsto \pi_{w^{-1} k w} \quad (k \in K(w))
\]
is an $R$-algebra isomorphism.
In particular, if $w \in {}^{J}W^{I}$ and $\kappa \in W_{K(w)}$, then
$$
\pi_{\kappa} \pi_w = \pi_w c_w(\pi_{\kappa}).
$$
\end{lemma}

\begin{proof}
To begin with, we note that $c_w$ restricts to a bijection from $\{\pi_k: k\in K(w)\}$ to 
$\{w^{-1}\pi_kw: k\in K(w)\}$.
Suppose that $w s w^{-1},  w s' w^{-1}  \in K(w)$ for some $s, s' \in I$. 
By the definition of $K(w)$, the elements $t := w i w^{-1}, t' := w i w^{-1}$ are contained in $J$.
For any $m \in \mathbb{N} \cup \{ \infty \}$, it holds that  
\begin{align*}
w^{-1}(t t' t \ldots )_m w = (s s' s \ldots)_m  \quad \text{ and }\quad 
w^{-1}(t' t t' \ldots )_m w = (s' s s' \ldots)_m, 
\end{align*}
hence the relation $(t t' t \ldots )_m = (t' t t' \ldots )_m$ gives rise to the relation    
$(s s' s \ldots)_m = (s' s s' \ldots)_m$.
In a direct way, one can also check that  
$c_w(\pi_s) c_w(\pi_s) = a c_w(\pi_s) +b$.
Furthermore, the inverse of $c_w$ can be obtained from the mapping $\pi_{\sigma} \mapsto \pi_{w \sigma w^{-1}}$. 
Putting the above discussions together, 
we can conclude that $c_w$ is an $R$-algebra isomorphism. 

For the second assertion, note that $J \subseteq Asc_{L}(w)$
and $\pi_l \pi_w = \pi_{l w}$ for each $l \in K(w)$. 
The former follows from Lemma~\ref{lem: minimal representatives of double coset} 
and the latter from the inclusion $K(w) \subseteq J$.
Therefore, one has that 
$\pi_{\kappa} \pi_w = \pi_{\kappa w}$
for all $\kappa \in W_K(w)$.
On the other hand, since $w \in W^I$ and $w^{-1} \kappa w \in W_I$, 
$w (w^{-1} \kappa w)$ is the reduced factorization of $\kappa w \in W$ with respect to $W/W_I$,
which induces the following equality:
\[\pi_{\kappa w} = \pi_{w} \pi_{w^{-1} \kappa w}.\]
Hence, letting $s_1s_2 \ldots s_l$ be any reduced expression for $\kappa$, 
it holds that 
\[
\pi_{w^{-1} \kappa w} = \pi_{w^{-1} s_1 w}\pi_{w^{-1} s_2 w} \cdots \pi_{w^{-1} s_l w} = c_w(\pi_{\kappa}),
\]
as required.
\end{proof}

Now we are in the position to state our main result.

\begin{theorem}\label{thm: Mackey}
Let $(W, S)$ be a finite Coxeter system, and $I, J \subseteq S$, and $M \in \module (H_I)$. Then
$$
M \big\uparrow^{H_S}_{H_I} \big\downarrow^{H_S}_{H_J} \,\,\cong \,\, \underset{\tau \in {}^JW^I}{\bigoplus} c_{\tau} [M \big\downarrow^{H_I}_{H_{{\tau}^{-1} K(\tau) \tau}}]\, \big\uparrow^{H_J}_{H_{K(\tau)}}
$$
as $H_J$-modules 
where
$c_{\tau}: H_{K(\tau)} \rightarrow H_{{\tau}^{-1} K(\tau) \tau}$ is the algebra homomorphism defined by $\pi_k \rightarrow \pi_{{\tau}^{-1} k \tau   }$.
\end{theorem}

\begin{proof}
For any $w \in W$,  by Lemma~\ref{lem: triple reduced expression}, there exists a unique triple $(z,\hat\tau, y)$ with   
$z \in W^{K(\tau)}_{J}, \hat \tau \in {}^{J}W^{I}, y \in W_I$ such that 
\begin{equation}\label{eq: triple factoriation of w}
w = z \hat\tau y \quad \text{ and } \quad \ell(w)=\ell(z)+\ell(\hat\tau)+\ell(y).
\end{equation}
For this factorization, it is clear that 
$\pi_{w} = \pi_{z} \pi_{\hat\tau} \pi_{y}.$
Consider the map 
\begin{align*}
\Phi: H_S \underset{H_I}{\otimes}M
&\rightarrow 
\underset{\tau \in {}^JW^I}{\bigoplus} 
H_J \underset{H_{K(\tau)}}{\otimes} c_{\tau}[M \big\downarrow^{H_I}_{H_{{\tau}^{-1} K(\tau) \tau}}]\\
\pi_w \otimes m &\mapsto E_{\hat\tau}(\pi_z \otimes \pi_y \cdot m) \quad (w \in W, \, m \in M),
\end{align*}
where $w=z \hat\tau y$ is the decomposition given by~\eqref{eq: triple factoriation of w}
and, for any $\tau\in {}^JW^I$ and $p\in H_J \underset{H_{K(\tau)}}{\otimes} c_{\tau}[M \downarrow^{H_I}_{H_{{\tau}^{-1} K(\tau) \tau}}]$, the notation $E_{\tau}(p)$ 
denotes the element in the codomain of $\Phi$
such that the $\tau$th entry is $p$ and the other entries are all zero, i.e., 
\[E_{\tau}(p):=(0, 0,\ldots, \stackunder{$\underbrace{p}$}{\scriptsize{$\tau$th}}, \ldots, 0,0).\]
We claim that $\Phi$ is well-defined, more precisely, 
\[\Phi(\pi_{w} \pi_i \otimes m)  = \Phi(\pi_{w} \otimes ( \pi_{i} \cdot m)) \quad \text{ for all } i\in I.\]
By Theorem~\ref{thm: left-right actions}, 
\begin{align*}
    \pi_y \pi_i =
    \begin{cases}\pi_{yi} &\text{ if } \ell(yi) > \ell(y),   \\
    a \pi_y + b \pi_{yi} &\text{ if } \ell(yi) < \ell(y)
\end{cases}
\end{align*}
for all $i \in I$.
If $\ell(yi) > \ell(y)$, then 
\begin{align*}
  \pi_{w} \pi_i \otimes m=\pi_{z} \pi_{\hat\tau} \pi_{y} \pi_{i} \otimes m
    = \pi_{z} \pi_{\hat\tau} \pi_{yi} \otimes m=
    \pi_{z} \pi_{\hat\tau} \otimes (\pi_{yi} \cdot m).
\end{align*}
Hence
\begin{align*}
    \Phi(\pi_{w} \pi_i \otimes m) 
    &= E_{\hat\tau}(\pi_{z} \otimes (\pi_{yi} \cdot m))\\
    &= E_{\hat\tau}(\pi_{z} \otimes (\pi_{y} \cdot (\pi_{i} \cdot m)))\\
    &= \Phi(\pi_{z} \pi_{\hat\tau} \pi_{y} \otimes ( \pi_{i} \cdot m))\\
    &= \Phi(\pi_{w} \otimes ( \pi_{i} \cdot m)).
\end{align*}
In a similar way as above, it can be seen that the same result holds 
when $\ell(yi) < \ell(y)$. Thus the claim is verified.

Next, we consider the map
\begin{align*}
\Psi: \underset{\tau \in {}^JW^I}{\bigoplus} 
 H_J \underset{H_{K(\tau)}}{\otimes} c_{\tau}[M \big\downarrow^{H_I}_{H_{{\tau}^{-1} K(\tau) \tau}}]
&\rightarrow
H_S \underset{H_I}{\otimes}M\\
 E_{\hat\tau}(\pi_{\xi} \otimes m) & \mapsto \pi_{\xi} \pi_{\tau}  \otimes m \quad (\xi \in W_J, \,m \in M).
\end{align*}
This map is also well defined since, for every $\kappa \in W_{K(\tau)}$,
\begin{align*}
    \Psi(E_{\tau}(\pi_{\xi} \pi_{\kappa} \otimes m)) &= \pi_{\xi} \pi_{\kappa} \pi_{\tau} \otimes m\\
    &= \pi_{\xi} \pi_{\tau} c_{\tau}(\pi_{\kappa}) \otimes m\\
    &= \pi_{\xi} \pi_{\tau} \otimes c_{\tau}(\pi_{\kappa}) \cdot m\\
    &= \Psi( E_{\tau}(\pi_{\xi} \otimes \pi_{\kappa } \cdot^{c_{\tau}} m)).
\end{align*}
Here the second equality follows from Lemma~\ref{lem: conjugate homomorphism}.
For the definition of $\cdot^{c_{\tau}}$ in the final term, see $\eqref{symbol for action}$.

We will show that $\phi$ and $\Psi$ are inverses to each other.
To do this, choose an arbitrary element $\xi$ in $W_J$. 
Let 
\begin{align}\label{eq: reduced expression for W_J}
\xi = z \kappa, \text{ where } z \in W^{K(\tau)}_{J} \text{ and } \kappa \in W_{K(\tau)},
\end{align}
be the reduced factorization of $\xi$ with respect to $W_J / W_{K(\tau)}$. 
Then $\pi_{\xi} = \pi_{z} \pi_{\kappa}$.
Using this factorization, we can show that, for any $m \in M$ and $\tau \in {}^{J}W^{I}$,   
\begin{align*}
   \Psi( E_{\tau}(\pi_{\xi} \otimes m)) = \pi_{\xi} \pi_{\tau} \otimes m
    = \pi_{z} \pi_{\kappa} \pi_{\tau} \otimes m
    = \pi_z \pi_{\tau} c_{\tau}(\pi_{\kappa}) \otimes m. 
\end{align*}
Here the third equality follows from Lemma~\ref{lem: conjugate homomorphism}.
Hence, 
\begin{align*}
    \Phi \circ \Psi( E_{\tau}(\pi_{\xi} \otimes m))
    &= E_{\tau}(\pi_z  \otimes (c_{\tau}(\pi_{\kappa}) \cdot m))\\
    &= E_{\tau}(\pi_z  \otimes \pi_{\kappa} \cdot^{c_{\tau}} m)\\
    &= E_{\tau}(\pi_z \pi_{\kappa}  \otimes m)\\
    &= E_{\tau}(\pi_{\xi}  \otimes m).
\end{align*}
In a similar manner as above, one can show that, for any $w \in W$ and $m \in M$, 
\begin{align*}
    \Psi \circ \Phi(\pi_w \otimes m) &= \Psi \circ \Phi(\pi_{z} \pi_{\hat\tau} \pi_{y} \otimes m)\\
    &= \Psi ( E_{\hat\tau}(\pi_{z} \otimes   (\pi_{y} \cdot m) ))\\
    &= \pi_{z} \pi_{\hat\tau} \otimes  \pi_{y} \cdot m \\
    &= \pi_{z} \pi_{\hat\tau} \pi_{y} \otimes m \\
    &= \pi_w \otimes m,
\end{align*}
where $w=z \hat\tau y$ is given by~\eqref{eq: triple factoriation of w}.
So we are done.

Finally, let us show that $\Psi$ is an $H_{J}$-module homomorphism.
Let $\xi \in W_J$, $m \in M$, and $\tau \in {}^{J}W^{I}$.
And let $\xi = z \kappa$ be the decomposition given by~\eqref{eq: reduced expression for W_J}.
By Theorem~\ref{thm: left-right actions}, 
\begin{align*}
    \pi_j \pi_z =
    \begin{cases}
    \pi_{jz} & \text{ if } \ell(jz) > \ell(z), \\
    a \pi_z + b \pi_{jz} & \text{ if } \ell(jz) < \ell(z)
     \end{cases}
\end{align*}
for $j \in J$.
Assume that $\ell(jz) > \ell(z)$.
Let 
\begin{align*}
jz = z' \kappa', \text{ where } z' \in W^{K(\tau)}_{J} \text{ and } \kappa' \in W_{K(\tau)},
\end{align*}
be the reduced factorization of $jz$ with respect to $W_J / W_{K(\tau)}$.
Then $\pi_{jz} = \pi_{z'} \pi_{\kappa'}$.
Using this factorization, we can show that
\begin{align*}
\pi_{j} E_{\tau}(\pi_{\xi} \otimes m) = E_{\tau}(\pi_{j} \pi_{z} \pi_{\kappa} \otimes m)
    &= E_{\tau}(\pi_{jz} \otimes \pi_{\kappa} \cdot^{c_{\tau}} m)\\
    &= E_{\tau}(\pi_{z'} \pi_{\kappa'}\otimes \pi_{\kappa} \cdot^{c_{\tau}} m)\\
    &= E_{\tau}(\pi_{z'} \otimes \pi_{\kappa'} \cdot^{c_{\tau}} (\pi_{\kappa} \cdot^{c_{\tau}} m)).
\end{align*}
Hence,
\begin{align*}
    \Psi(\pi_{j} E_{\tau}(\pi_{\xi} \otimes m)) 
    &= \Psi( E_{\tau}(\pi_{z'} \otimes \pi_{\kappa'} \cdot^{c_{\tau}} (\pi_{\kappa} \cdot^{c_{\tau}} m)) )\\
    &= \pi_{z'} \pi_{\tau} \otimes c_{\tau}(\pi_{\kappa'}) \cdot ( c_{\tau}(\pi_{\kappa}) \cdot m)\\
    &= \pi_{z'} \pi_{\tau} c_{\tau}(\pi_{\kappa'}) \otimes  (c_{\tau}(\pi_{\kappa}) \cdot m).
\end{align*}
On the other hand, we can show that
\begin{align*}
    \pi_{j} \Psi( E_{\tau}(\pi_{\xi} \otimes m)) 
    &= \pi_{j} \Psi( E_{\tau}(\pi_z \otimes \pi_{\kappa} \cdot^{c_{\tau}} m))\\
    &= \pi_{j} (\pi_z \pi_{\tau} \otimes (c_{\tau}(\pi_{\kappa}) \cdot m))\\
    &= \pi_{z'} \pi_{\kappa'} \pi_{\tau} \otimes (c_{\tau}(\pi_{\kappa}) \cdot m).
\end{align*}
By Lemma~\ref{lem: conjugate homomorphism}, we conclude that $\Psi(\pi_{j} E_{\tau}(\pi_{\xi} \otimes m)) = \pi_{j} \Psi( E_{\tau}(\pi_{\xi} \otimes m)) $.
In the case of $\ell(jz) < \ell(z)$, we have $\pi_j \pi_z = a \pi_z + b \pi_{jz}$. Let $\pi_{z} = \pi_{z''} \pi_{\kappa''}$ and $\pi_{jz} = \pi_{z'''} \pi_{\kappa'''}$ be the reduced factorizations of $z$ and $jz$ with respect to $W_J / W_{K(\tau)}$, respectively. In a similar way as above, it can be seen that the same result holds when $\ell(jz) < \ell(z)$.

This completes the proof.
\end{proof}

\begin{remark}
Lemma~\ref{lem: conjugate homomorphism} plays a key role in the proof of Theorem~\ref{thm: Mackey}.
By extending this lemma to Iwahori-Hecke algbras, one can state Theorem~\ref{thm: Mackey}
for Iwahori-Hecke algebras without difficulty.
The Mackey decomposition formula thus obtained 
recovers~\cite[Proposition 9.1.8]{GP00} 
since the latter is stated only in the case where 
given parameters $a_s, b_s\in R$
satisfy that $b_s=1-a_s$ for all $s \in S$ 
and the $H_R(W,S,\{a_s, 1-a_s: s\in S\})$-module in consideration is 
free as an $R$-module.
\end{remark}

Let's take a closer look at the case where $R = \mathbb{C}$, $W= \SG_{m+n}$, $W_I$ and $W_J$ are maximal parabolic subgroups of $W$, and $M$ is a tensor product of two modules. 
In the following, we simply write $H_n$ for $H_{\SG_n}(a,b)$.
Then Theorem~\ref{thm: Mackey} can be rewritten in the following simple form.

\begin{corollary}\label{cor: Mackey for 0-Hecke}
For all $a,b\in \mathbb C$ and $1\le k \le m+n-1$, we have the following isomorphism of $H_k \otimes H_{m+n - k}$-modules: 
for $M \in \module H_m$ and $N \in \module H_n$,
\begin{align*}
(M \boxtimes N) \big\downarrow_{H_k \otimes H_{m+n - k}}^{H_{m+n}}
\cong 
\hspace{-0.8ex}\bigoplus_{\substack{t + s = k \\ t \le m,~ s \le n }} \hspace{-0.8ex}
\mathtt{T}_{t,s}
\left(
M\big\downarrow^{H_m}_{H_t \otimes H_{m-t}} \otimes \;
N\big\downarrow^{H_n}_{H_s \otimes H_{n-s}}
\right)
\big\uparrow^{H_k \otimes H_{m+n-k}}_{H_t \otimes H_s \otimes H_{m-t} \otimes H_{n-s}},
\end{align*}
where 
$M \boxtimes N= M \otimes N \big\uparrow_{H_m \otimes H_n}^{H_{m+n}}$ and
\[
\mathtt{T}_{t,s}: \module (H_t \otimes H_{m-t} \otimes H_s \otimes H_{n-s})
\ra \module (H_t \otimes H_s \otimes H_{m-t} \otimes H_{n-s})
\]
is the functor sending $M_1 \otimes M_2 \otimes N_1 \otimes N_2 \mapsto M_1 \otimes N_1 \otimes M_2 \otimes N_2$.
\end{corollary}
\begin{proof}
Let $I=\{1,2,\ldots,m-1,m+1,m+2,\ldots,m+n-1\}$ and $J=\{1,2,\ldots,k-1,k+1,k+2,\ldots,m+n-1\}$.
It is well known that 
$$
W^I = \{w \in \SG_{m+n} \mid w(1) < \cdots < w(m) \ \text{and} \ w(m+1) < \cdots < w(m+n) \}
$$
(for instance, see~\cite{05BB}).
Combining ${}^J W =\{w \mid w^{-1} \in W^J\}$ with Lemma~\ref{lem: minimal representatives of double coset}, we derive that 
\[{}^{J}W^{I} = \{w_t \mid 0 \le t \le m, 0 \le k-t \le n \},\] 
where
\begin{align*}
    w_t(i) = 
    \begin{cases}
    i   & \text{ if } 1 \le i \le t, \\
    k-t+i   & \text{ if } t+1 \le i \le m, \\
    t-m+i   & \text{ if } m+1 \le i \le m+k-t, \\
    i   & \text{ if } m+k-t+1 \le i \le m+n,
    \end{cases}
\end{align*}
i.e.,
$$
w_t=1 \ldots t \mid k+1 \ldots k+m-t \mid t+1 \ldots k \mid m+k-t+1 \ldots m+n 
$$
in one-line notation.
It says that the functor $F_{c_{w_{t}}}$ induced by $c_{w_{t}}$ and the functor $\mathtt{T}_{t,s}$ are the same,
thus the assertion follows.
\end{proof}

\begin{remark}
(1) Recall that $H_n(1,0)=H_n(0)$, the $0$-Hecke algebra of $\mathfrak S_n$.
Hence Corollary~\ref{cor: Mackey for 0-Hecke} 
is a generalization of \cite[Thoerem 3]{21JKLO}.
In fact, the latter is stated only for the $0$-Hecke modules called Weak Bruhat interval modules.

(2) Consider tower of algebras $A = \bigoplus_{n \ge 0} A_n$ with $\dim(A_1) = 1$ 
such that Grothendieck groups $G(A)$ and $K(A)$ form graded dual Hopf algebras.
Bergeron, Lam, and Li conjectured in \cite[Conjecture 6.2]{BLL12} that such a tower of algebra 
is isomorphic to a tower $H(a,b) = \bigoplus_{n \ge 0} H_n(a,b)$ of algebras for some $a,b \in \mathbb{C}$.
Thus, under the validity of the conjecture, we can state the following isomorphism of 
$A_k \otimes A_{m+n - k}$-modules:
for $M \in \module A_m$ and $N \in \module A_n$,
\begin{align*}
&(M \boxtimes N) \big\downarrow_{A_k \otimes A_{m+n - k}}^{A_{m+n}}
\cong 
\hspace{-0.8ex}\bigoplus_{\substack{t + s = k \\ t \le m,~ s \le n }} \hspace{-0.8ex}
\mathtt{T}_{t,s}
\left(
M\big\downarrow^{A_m}_{A_t \otimes A_{m-t}} \otimes \;
N\big\downarrow^{A_n}_{A_s \otimes A_{n-s}}
\right)
\big\uparrow^{A_k \otimes A_{m+n-k}}_{A_t \otimes A_s \otimes A_{m-t} \otimes A_{n-s}},
\end{align*}
where 
$M \boxtimes N= M \otimes N \big\uparrow_{A_m \otimes A_n}^{A_{m+n}}$ and
\[
\mathtt{T}_{t,s}: \module (A_t \otimes A_{m-t} \otimes A_s \otimes A_{n-s})
\ra \module (A_t \otimes A_s \otimes A_{m-t} \otimes A_{n-s})
\]
is the functor sending $M_1 \otimes M_2 \otimes N_1 \otimes N_2$ to $M_1 \otimes N_1 \otimes M_2 \otimes N_2$.
\end{remark}


\section{(Anti-)involutions of two variable Hecke algebras and their interaction with induction product and restriction}
\label{Sec. Interaction with induction product and restriction}
Recall that the induced product of modules is given by $M \boxtimes N= M \otimes N \big\uparrow_{H_m \otimes H_n}^{H_{m+n}}$.
In this section, we investigate how the induction product $\boxtimes$ or the restriction $\downarrow^{H_{m+n}}_{H_m \otimes H_n}$ are intertwined with several (anti-)automorphisms twists of the two variable Hecke algebra $H_{m+n}$.
Recently, this subject was considered in~\cite[Corollary 1]{21JKLO} only for a class of modules of $0$-Hecke algebras called Weak Bruhat interval modules.
The content of this section can be seen as a broad generalization of~\cite[Corollary 1]{21JKLO}.

In~\cite{05Fayers}, Fayers introduced the involutions $\autotheta$ and $\autophi$ and the anti-involution $\autochi$ of the $0$-Hecke algebra $H_W(0)$ defined in the following manner:
\begin{align*}
&\autophi: H_W(0) \ra H_W(0), \quad \pi_s \mapsto \pi_{w_0 s w_0} \quad \text{for $s \in S$},\\
&\autotheta: H_W(0) \ra H_W(0), \quad \pi_s \mapsto  1 - \pi_s \quad \text{for $s \in S$}, \\
&\autochi: H_W(0) \ra H_W(0), \quad \pi_s \mapsto \pi_s \quad \text{for $s \in S$}.
\end{align*}
We are going to extend these to the morphisms of the two variable Hecke algebras.

\begin{lemma}\label{lem: theta braid}
For fixed $a, b \in R$, let $(y_n)_{n \ge 0}$ be the sequence in $R$ determined by the following recurrence
\[
y_n = -a y_{n-1} + b y_{n-2}, \quad y_0 = 0, \quad y_1 = -a.
\]
Then the following equality holds.
\[
( (\pi_i - a)(\pi_j - a)(\pi_i - a)\ldots )_n = (\pi_i \pi_j \pi_i \ldots)_n + \sum_{m=1}^{n-1} y_{n-m} ( (\pi_i \pi_j \pi_i \ldots)_m + (\pi_j \pi_i \pi_j \ldots)_m ) + y_n.
\]
\end{lemma}
\begin{proof}
Use induction on $n$.
\end{proof}

\begin{lemma}\label{involutions and anti-involution}
Let $(W,S)$ be a Coxeter system.
There are automorphisms $\autophi, \autotheta$ and an anti-automorphism $\autochi$ on $H_W$ defined by
\begin{align*}
&\autophi: H_W \ra H_W, \quad \pi_s \mapsto  \pi_{w_0 s w_0} \quad \text{for $s \in S$},\\
&\autotheta: H_W \ra H_W, \quad \pi_s \mapsto a-\pi_s \quad \text{for $s \in S$}, \\
&\autochi: H_W \ra H_W, \quad \pi_s \mapsto \pi_s \quad \text{for $s \in S$}.
\end{align*}
Furthermore, they commute with each other and are involutions.
\end{lemma}
\begin{proof}
It is clear that $\autophi^2, \autotheta^2$, and $\autochi^2$ are identity maps and they commute with each other, and $\autochi$ is an anti-automorphism.
The map $\autophi$ is an automorphism due to Lemma~\ref{lem: conjugate homomorphism}.
To show $\autotheta$ is an automorphism, observe $(a-\pi_i)^2 = a^2 - 2 a \pi_i + (\pi_i)^2 = a(a-\pi_i) +b$.
The braid relation follows from Lemma~\ref{lem: theta braid}.
\end{proof}

Let $\opi_i := \pi_i - a$. Then for any reduced expression $s_{i_1}s_{i_2} \ldots s_{i_l}$ of $w \in W$, $\opi_w := \opi_{i_1}\opi_{i_2} \ldots \opi_{i_l}$ is independent of choices of reduced expressions since $\opi_i$ satisfy the braid relations.
Similar to Theorem~\ref{thm: left-right actions}, one can show that two variable Hecke algebra $H_W(a, b)$ over $R$ is a free $R$-module with basis elements $\opi_w \, (w \in W)$, and for all $s \in S, w \in W$, multiplication is given by
\[\opi_s \opi_w
=\begin{cases}
\opi_{sw}          &\text{ if } \ell(sw)>\ell(w), \\ 
-a\opi_w+ b\opi_{sw} & \text{ if }\ell(sw)<\ell(w).
\end{cases}\]

\subsection{Formulas for the involution twists and the induction product and formulas for the restriction}
In the following, let us investigate how (anti-)involution twists behave with respect to induction products and restrictions.

\begin{lemma}\label{lem: ind, rest of algebra modules}
Let $B$ be a subalgebra of $A$ and $\alpha$ be an automorphism of $A$.
\begin{enumerate}[label = {\rm (\arabic*)}]
    \item For a $B$-module $K$,
    $\alpha[K \uparrow^A_B] \cong \alpha[K] \uparrow^A_{\alpha^{-1}(B)}$.
    \item For an $A$-module $L$,
    $\alpha[L \big\downarrow^A_B] \cong \alpha[L] \big\downarrow^A_{\alpha^{-1}(B)}$.
\end{enumerate}
\end{lemma}
\begin{proof}
(1) Note that $\alpha |_{\alpha^{-1}(B)}: \alpha^{-1}(B) \to B$ is an isomorphism, so $\alpha[K]$ is an $\alpha^{-1}(B)$-module.
Consider the bijection $\Phi: \alpha[A \otimes_B K] \to \alpha[A] \otimes_{\alpha^{-1}(B)} \alpha[K]$ given by
$a \otimes k \mapsto a \otimes k$ for any $a \in A$ and $k \in K$.
Since $\Phi(a \cdot b \otimes k) = a \cdot^{\alpha} \alpha^{-1}(b) \otimes k = a \otimes \alpha^{-1}(b) \cdot_\alpha k = a \otimes b \cdot k = \Phi (a \otimes b \cdot k )$ for any $b \in B$, $\Phi$ is well-defined.
The $A(=\alpha(A))$-module structure on the left hand side is given by 
$a' \cdot^{\alpha} (a \otimes k) = (\alpha(a') \cdot a) \otimes k$ for any $a', a \in A$ and $k \in K$.
And, $a' \cdot^{\alpha} \Phi( a \otimes k) = a' \cdot^{\alpha} ( a \otimes k) = (\alpha(a') \cdot a) \otimes k$, which shows that $\Phi$ is an $A$-module homomorphism.
(2) Similarly, the identity map on the vector space $L$ is an isomorphism between two $\alpha^{-1}(B)$-modules.
\end{proof}

Let $\autoomega:= \autophi \circ \autotheta$.
Then we can derive the following relations.

\begin{theorem}\label{thm: automorhpism, induction, restriction}
Let $M \in \module H_m(0)$, $N \in \module H_n(0)$. Then we have following isomorphisms of $H_{m+n}(0)$-modules:
\begin{enumerate}[label = {\rm (\arabic*)}]
    \item $\autophi[M \boxtimes N] \cong \autophi[N] \boxtimes \autophi[M]$
    \item $\autotheta[M \boxtimes N] \cong \autotheta[M] \boxtimes \autotheta[N]$
    \item $\autoomega [M \boxtimes N] \cong \autoomega [N] \boxtimes \autoomega [M]$
\end{enumerate}    
For $L \in \module H_{m+n}$, we have following isomorphisms of $H_m \otimes H_n$-modules:
\begin{enumerate}[resume*]
    \item $\autophi[L \downarrow^{H_{m+n}}_{H_n \otimes H_m}] \cong \autophi[L] \downarrow^{H_{m+n}}_{H_m \otimes H_n}$
    \item $\autotheta[L \downarrow^{H_{m+n}}_{H_m \otimes H_n}] 
\cong \autotheta[L] \downarrow^{H_{m+n}}_{H_m \otimes H_n}$
    \item If $L$ is a free $R$-module of finite rank, then 
    $
    \autochi[L \downarrow^{H_{m+n}}_{H_m \otimes H_n}] \cong \autochi[L] \downarrow^{H_{m+n}}_{H_m \otimes H_n}.
    $
\end{enumerate}
\end{theorem}
\begin{proof}
(1) Put $A=H_{m+n}$, $B=H_m \otimes H_n$ and $\alpha = \autophi$ in Lemma~\ref{lem: ind, rest of algebra modules} (1), then we have following isomorphisms:
\begin{align*}
\autophi[M \boxtimes N]
&\cong \autophi[M \otimes N \uparrow^{H_{m+n}}_{H_m \otimes H_n}]\\
&\cong \autophi[M \otimes N] \uparrow^{H_{m+n}}_{\autophi^{-1}(H_m \otimes H_n)}\\
&\cong (\autophi[M] \otimes \autophi[N]) \uparrow^{H_{m+n}}_{\autophi^{-1}(H_m \otimes H_n)}\\
&\cong (\autophi[M] \otimes \autophi[N]) \uparrow^{H_{m+n}}_{H_n \otimes H_m}\\
&\cong \autophi[N] \boxtimes \autophi[M]
\end{align*}

(2) Replace $\autophi$ with $\autotheta$ in the above.

(3) Combine (1) and (2).

(4) \& (5) Immediately follows from Lemma~\ref{lem: ind, rest of algebra modules} (2).

(6) Let $\{e_i\}$ and $\{\epsilon_i\}$ be dual basis for $L$ and $\autochi[L]$ so that if $\langle,\rangle$ is the bilinear form given by $\langle e_i, \epsilon_j \rangle = \delta_{ij}$, then $\langle \pi_i \cdot m, \mu \rangle = \langle m, \pi_i \cdot \mu \rangle$ for all $m \in M, \mu \in \autochi[M]$ and $i \in \{1,2,\ldots,m+n\}$.
Then we automatically have $\langle \pi_i \cdot m, \mu \rangle = \langle m, \pi_i \cdot \mu \rangle$ for all $i \in \{1,2,\ldots,m+n\} \setminus \{m\}$, which proves our assertion.
\end{proof}

\subsection{Formulas for the anti-involution twists and the induction product}
In this subsection, we will show how the anti-involution twists interact with induction product.
For $W = \SG_{m+n}$ and $W_I = \SG_m \times \SG_n$, let us denote by $\Gamma$ the set of minimal length left coset representatives $W^I$.

\begin{lemma}\label{lem: dimensions of induced modules}
Let $M \in \module(H_m \otimes H_n)$ be a free $R$-module of finite rank and $\beta$ be a basis for $M$. Then $\{ \pi_{\gamma} \otimes \beta_i \mid \gamma \in \Gamma, \beta_i \in \beta \}$ and $\{ \opi_{\gamma} \otimes \beta_i \mid \gamma \in \Gamma, \beta_i \in \beta \}$ are bases for $M \uparrow ^{H_{m+n}}_{H_m \otimes H_n}$.
\end{lemma}
\begin{proof}
By the reduced factorization of each element in $\SG_{m+n}$ with respect to $\SG_{m+n} / (\SG_m \times \SG_n)$, $\{ \pi_{\gamma} \otimes \beta_i \}$ spans $M \uparrow ^{H_{m+n}}_{H_m \otimes H_n}$. Since $\dim(M \uparrow ^{H_{m+n}}_{H_m \otimes H_n}) = | \binom{m+n}{m} | \cdot \dim(M)$ and $ | \binom{m+n}{m} | = | \Gamma |$, $\{ \pi_{\gamma} \otimes \beta_i \}$ is a basis for $M \uparrow ^{H_{m+n}}_{H_m \otimes H_n}$. Using the same argument, we can see that $\{ \opi_{\gamma} \otimes \beta_i \}$ is also a basis for $M \uparrow ^{H_{m+n}}_{H_m \otimes H_n}$.
\end{proof}

It is well known that 
\[\Gamma = \{\gamma \in \SG_{m+n} \mid \gamma(1) < \cdots < \gamma(m) \ \text{and} \ \gamma(m+1) < \cdots < \gamma(m+n) \}.\] 
For later use, we simply write $\gamma \in \Gamma$ as 
\[
\gamma(1) < \cdots < \gamma(m) \mid \gamma(m+1) < \cdots < \gamma(m+n)
\]
(refer to~\cite[Lemma 2.4.7]{05BB}).
Dividing cases according to where $i, i+1$ appear  in the one-line notation of $\gamma$, we obtain the following equalities for any $i \in \{1,2,\ldots,m+n-1 \}$ and $\gamma \in \Gamma$:
\begin{align}\label{eq: minimal decomposition unbar}
    \pi_i  \pi_{\gamma}= 
\begin{cases}
    \pi_{\gamma}  \pi_{\gamma^{-1}(i)}& \text{if } \gamma^{-1}(i) \le m, \gamma^{-1}(i+1) \le m,\\
    \pi_{\gamma}  \pi_{\gamma^{-1}(i)}& \text{if } \gamma^{-1}(i) > m, \gamma^{-1}(i+1) > m,\\
    \pi_{s_i  \gamma}& \text{if } \gamma^{-1}(i) \le m, \gamma^{-1}(i+1) > m,\\
    a \pi_{\gamma} + b \pi_{s_i  \gamma}& \text{if } \gamma^{-1}(i) > m, \gamma^{-1}(i+1) \le m.
\end{cases}
\end{align}
Similarly, we obtain the equalities for any $i \in [m+n-1]$ and $\gamma \in \Gamma$:
\begin{align}\label{eq: minimal decomposition bar}
    \pi_i  \opi_{\gamma}= 
\begin{cases}
    \opi_{\gamma}  \pi_{\gamma^{-1}(i)}& \text{if } \gamma^{-1}(i) \le m, \gamma^{-1}(i+1) \le m,\\
    \opi_{\gamma}  \pi_{\gamma^{-1}(i)}& \text{if } \gamma^{-1}(i) >m, \gamma^{-1}(i+1) > m,\\
    \opi_{s_i  \gamma} + a \opi_{\gamma}& \text{if } \gamma^{-1}(i) \le m, \gamma^{-1}(i+1) > m,\\
    b \opi_{s_i  \gamma}& \text{if } \gamma^{-1}(i) > m, \gamma^{-1}(i+1) \le m.
\end{cases}
\end{align}
Note that all $s_i  \gamma$'s appearing in~\eqref{eq: minimal decomposition unbar} and ~\eqref{eq: minimal decomposition bar} are in $\Gamma$.

Let $\hautophi:= \autophi \circ \autochi$, $\hautotheta := \autotheta \circ \autochi$, $\hautoomega:= \autoomega \circ \autochi$.
With this notation, we can state the main result of this subsection.

\begin{theorem}\label{thm: chi twist, induction}
Let $M \in  \module H_m$ and $N \in \module H_n$ be free $R$-modules of finite rank. Then we have following isomorphisms of $H_{m+n}$-modules:
\begin{enumerate}[label = {\rm (\arabic*)}]
    \item $\widehat{\autophi}[M \boxtimes N] \cong 
\widehat{\autophi}[M] \boxtimes \widehat{\autophi}[N]$
    \item $\autochi[M \boxtimes N] \cong \autochi[N] \boxtimes \autochi[M]$
    \item $\hautotheta[M \boxtimes N] \cong \hautotheta[N] \boxtimes \hautotheta[M]$
    \item $\hautoomega[M \boxtimes N] \cong \hautoomega[M] \boxtimes \hautoomega[N]$
\end{enumerate}  
\end{theorem}

\begin{proof} (1) 
There exist bases $\{ e_k^M: k=1,2, \ldots, {\rm rank}(M)\}$ and $\{\epsilon_l^M : l=1,2, \ldots, {\rm rank}(M)\}$ for  $M$ and $\autochi[\autophi[M]]$, respectively, and a bilinear pairing $\langle \, , \,\rangle_M$ such that 
$\langle e_k^M, \epsilon_l^M \rangle_M = \delta_{kl}$ and $\langle \pi_i \cdot x, y \rangle_M = \langle x, \pi_{m-i} \cdot y \rangle_M$ for $x\in M, y\in \autochi[\autophi[N]]$, and $1\le i\le m-1$.
Similarly, there exist bases $\{ e_k^N: k=1,2, \ldots, {\rm rank}(N) \}$ and $\{ \epsilon_l^N: l=1,2, \ldots, {\rm rank}(N) \}$ for $N$ and $\autochi[\autophi[N]]$, respectively, and a bilinear pairing $\langle \,, \,\rangle_N$ such that $\langle e_k^N, \epsilon_l^N \rangle_N = \delta_{kl}$ and $\langle \pi_i \cdot x, y\rangle_N = \langle x, \pi_{m+n-i} \cdot y \rangle_N$ for all $x \in N, y \in \autochi[\autophi[N]]$ and $m+1\le i \le  m+n-1$. 
Using this, one can deduce that there exist bases $\{ e_k \}$ and $\{ \epsilon_l \}$ for $M \otimes N$ and $\autochi[\autophi[M]] \otimes \autochi[\autophi[N]]$, respectively, and a bilinear pairing 
$\langle \,,\, \rangle: M \otimes N \times \autochi[\autophi[M]] \otimes \autochi[\autophi[N]] \ra \mathbb C$ such that 
$\langle e_k, \epsilon_l \rangle = \delta_{kl}$
and 
\begin{align*}
    \langle \pi_i \cdot z, \zeta \rangle=
    \begin{cases}
    \langle z, \pi_{m-i} \cdot \zeta \rangle,   &\text{if } 1 \le i \le m-1\\
    \langle z, \pi_{m+n-i} \cdot \zeta \rangle, &\text{if } m+1 \le i \le m+n-1
    \end{cases}
\end{align*}
for all $k,l$ and $z \in M \otimes N, \zeta \in \autochi[\autophi[M]] \otimes \autochi[\autophi[N]]$ for $i \in \{1,2,\ldots,m+n-1 \}$.
From Lemma~\ref{lem: dimensions of induced modules} it follows that $\{ \pi_{\gamma} \otimes e_k \}$ and $\{ \opi_{\gamma} \otimes \epsilon_l \}$ are bases for $M \otimes N \uparrow^{H_{m+n}}_{H_m \otimes H_n}$ and $\autochi[\autophi[M]] \otimes \autochi[\autophi[N]]) \uparrow^{H_{m+n}}_{H_m \otimes H_n}$, respectively.
For 
\[\gamma = \gamma(1) < \cdots < \gamma(m) \mid \gamma(m+1) < \ldots < \gamma(m+n) \in \Gamma,\] 
we let 
\begin{align*}
 \gamma' := \gamma'(1) < \cdots < \gamma'(m) \mid \gamma'(m+1) < \ldots < \gamma'(m+n) \in \Gamma,
\end{align*}
where 
$$\gamma'(i) =  
\begin{cases}
     (m+n+1)-\gamma(m+1-i) & \text{ if } 1\le i \le m,\\   
     (m+n+1) - \gamma(2m+n+1-i)& \text{ if } m+1\le i \le m+n.
\end{cases} 
$$
Clearly the assignment $\gamma \mapsto \gamma'$ induces 
an involution on the set $\Gamma$. 
We now define a bilinear pairing $( \,,\,): M \boxtimes N \times \autochi[\autophi[M]] \boxtimes \autochi[\autophi[N]] \ra \mathbb C$ by letting 
\[
(\pi_{\gamma} \otimes z, \opi_{\lambda} \otimes \zeta):= \delta_{\gamma \lambda'}  \langle z, \zeta \rangle
\]
for $z \in M \otimes N, \zeta \in \autochi[\autophi[M]] \otimes \autochi[\autophi[N]]$, and $\gamma, \lambda \in \Gamma$.
For the assertion, we have only to show that  \begin{align}\label{eq: pairing}
    (\pi_i  \pi_{\gamma} \otimes e_k, \opi_{\lambda} \otimes \epsilon_l) 
= (\pi_{\gamma} \otimes e_k, \pi_{m+n-i}  \opi_{\lambda} \otimes \epsilon_l)
\end{align}
for all $k, l$, and $i \in \{1,2,\ldots,m+n-1\}$, and $\gamma, \lambda \in \Gamma$.
Due to~\eqref{eq: minimal decomposition unbar}, 
the left hand side of \eqref{eq: pairing} is given as follows:

\begin{enumerate}[label = {\rm (A\arabic*)}]
\item If $\gamma^{-1}(i) \le m$ and $\gamma^{-1}(i+1) \le m$, then 
\[
\begin{cases}
\langle \pi_{\gamma^{-1}(i)} \cdot e_k,  \epsilon_l \rangle & \text{if $\lambda = \gamma'$,} \\
      0 & \text{otherwise.}
\end{cases}
\]

\item
If $\gamma^{-1}(i) >m$ and $\gamma^{-1}(i+1) > m$, then 
\[
\begin{cases}
\langle \pi_{\gamma^{-1}(i)} \cdot e_k,  \epsilon_l \rangle & \text{if  $\lambda = \gamma'$,} \\
  0 & \text{otherwise.}
\end{cases}
\]

\item
If $\gamma^{-1}(i) \le m$ and $\gamma^{-1}(i+1) > m$, then 
\[
\begin{cases}
   \langle e_k,  \epsilon_l \rangle & \text{if  $\lambda = (s_i \cdot \gamma)'$,} \\
    0 & \text{otherwise.}
\end{cases}
\]

\item
If $ \gamma^{-1}(i) > m$ and $\gamma^{-1}(i+1) \le m$, then 
\[
\begin{cases}
       a \langle e_k,  \epsilon_l \rangle & \text{if  $\lambda = \gamma'$,} \\
       b \langle e_k,  \epsilon_l \rangle & \text{if  $\lambda = (s_i \cdot \gamma)'$,} \\
       0 & \text{otherwise.}
\end{cases}
\]
\end{enumerate}
On the other hand, due to~\eqref{eq: minimal decomposition bar}, the right hand side of \eqref{eq: pairing} is given as follows: 

\begin{enumerate}[label = {\rm (B\arabic*)}]
\item If $\lambda^{-1}(m+n-i) \le m$ and $ \lambda^{-1}(m+n-i+1) \le m$, then 
\[
\begin{cases}
 \langle e_k, \pi_{\lambda^{-1}(m+n+i)} \cdot  \epsilon_l \rangle & \text{if  $\gamma = \lambda'$,} \\0 & \text{otherwise.}
\end{cases}
\]

\item
If $\lambda^{-1}(m+n-i) >m$ and $ \lambda^{-1}(m+n-i+1) > m$, then 
\[
\begin{cases}
 \langle e_k, \pi_{\lambda^{-1}(m+n+i)} \cdot  \epsilon_l \rangle & \text{if  $\lambda = \gamma'$,} \\
      0& \text{otherwise.}
\end{cases}
\]

\item
If $\lambda^{-1}(m+n-i) \le m$ and $ \lambda^{-1}(m+n-i+1) > m$, then 
\[
\begin{cases}
 a \langle e_k,  \epsilon_l \rangle & \text{if $\lambda=\gamma'$,} \\
      \langle e_k,  \epsilon_l \rangle & \text{if $s_{m+n-i} \cdot \lambda = \gamma'$,} \\
      0 & \text{otherwise.}
\end{cases}
\]

\item
If $ \lambda^{-1}(m+n-i) > m$ and $ \lambda^{-1}(m+n-i+1) \le m$, then 
\[
\begin{cases}
       b \langle e_k,  \epsilon_l \rangle & \text{if $s_{m+n-i} \cdot \lambda = \gamma'$,} \\
      0, & \text{otherwise.}
\end{cases}
\]
\end{enumerate}
Comparing these calculations, one can deduce the equality~\eqref{eq: pairing}.
For instance, in the case where $\gamma^{-1}(i) \le m$,  $\gamma^{-1}(i+1) > m$, and $\lambda = (s_i \cdot \gamma)'$,
one has that $\lambda^{-1}(m+n-i) \le m$ and $ \lambda^{-1}(m+n-i+1) > m$ and $s_{m+n-i} \cdot \lambda = \gamma'$.
So, in this case, both sides of~\eqref{eq: pairing} are equal to $\langle e_k, \epsilon_l \rangle$ due to (A3) and (B4).

(2) Replace $M$ by $\autophi(M)$ and $N$ by $\autophi(N)$ in (1). Then the assertion follows from Theorem~\ref{thm: automorhpism, induction, restriction} (1). 

(3) \& (4) Combine (2) with Theorem~\ref{thm: automorhpism, induction, restriction} (1) and (2).
\end{proof}

\begin{remark}
The isomorphism 
\[\widehat{\autophi}[M] \uparrow^{H_n(0)}_{H_{n-1}(0)} \cong \widehat{\autophi}[M \uparrow^{H_n(0)}_{H_{n-1}(0)}]\]
has already appeared in~\cite[Lemma~6.4]{05Fayers}. 
The proof of Theorem~\ref{thm: chi twist, induction} can be viewed as a generalization of Fayers' proof.
\end{remark}

\bibliographystyle{plain}
\bibliography{references}
\end{document}